\documentclass[11pt,a4paper]{amsart}

\usepackage{amssymb}
\usepackage{amsthm}
\usepackage{amsmath}
\usepackage{mathrsfs}

\def\tych{{3\hspace{-.1em}{\frac{\raise-.5pt\hbox{$\scriptscriptstyle 1$}}{\raise2pt\hbox{$\scriptscriptstyle 2$}}}}}

\newtheorem{theorem}{Theorem}
\newtheorem*{theoremA}{Theorem A}
\newtheorem*{theoremB}{Theorem B}
\newtheorem*{theoremC}{Theorem C}
\newtheorem*{theoremD}{Theorem D}
\newtheorem*{corollary*}{Corollary} 
\newtheorem{proposition}{Proposition} 

\theoremstyle{definition} 
\newtheorem{definition}{Definition}
\newtheorem{example}{Example}
\newtheorem*{problem*}{Problem}

\theoremstyle{remark}
\newtheorem{remark}{Remark}

\begin{document}

\title{Free Topological Mal'tsev Algebras}

\author{Ol'ga~V.~Sipacheva}

\email{ovsipa@gmail.com}

\author{Aleksandr~A.~Solonkov}

\email{456700@bk.ru}

\address{Department of General Topology and Geometry, Faculty of Mechanics and Mathematics, 
Lomonosov Moscow State University, Moscow, Russia}

\begin{abstract}
A brief introduction to universal algebra and the theory of topological algebras, their varieties, 
and free topological algebras is presented. Free topological Mal'tsev algebras are studied. Their properties, 
relationship with topological groups and heaps, and the problem of direct limit decomposition are considered. 
\end{abstract}



\maketitle

In 1954 \cite{Maltsev1954} Mal'tsev proved that a variety $\mathscr V$ of universal algebras is 
\emph{congruence-permutable} (i.e., any two congruences of any algebra in $\mathscr V$ commute under composition) 
if and only if there exists a term $\mu$ with three variables  
in the language of $\mathscr V$ such that the equations 
$$
x=\mu(x,y,y)=\mu(y,y,x)\eqno(\mathrm M)
$$ 
hold identically in $\mathscr V$. This theorem had a great impact on the further development of universal 
algebra. 

In this paper, we refer to $\mu$ as a \emph{Mal'tsev term} and to an operation satisfying condition 
$(\mathrm M)$ as a \emph{Mal'tsev operation} and study \emph{topological Mal'tsev algebras}, that is, topological 
spaces with a continuous Mal'tsev operation. We begin with a brief introduction to the theory of topological 
universal algebras, their varieties, and free topological algebras, prove some fundamental properties of 
topological Mal'tsev algebras, and study in detail free topological algebras in the full variety of topological 
Mal'tsev algebras. We also consider a relationship between topological Mal'tsev algebras and topological heaps 
and discuss conditions under which a topological space is a retract of a topological group and some related 
questions.

It turns out that Mal'tsev's theorem remains valid for varieties of \emph{topological} algebras. The following 
theorem readily follows from results of Mal'tsev and Taylor; its proof is given in Section~\ref{sec2}.

\begin{theoremA}
A variety of topological algebras is congruence-permutable if and only if there exists a Mal'tsev term in 
the language of this variety.
\end{theoremA}

The permutability of congruences in a variety of topological algebras is also equivalent to following important 
property, which makes it possible to form topological quotients of algebras in that variety. Its proof is given 
in Section~\ref{sec2}. 

\begin{theoremB}
If a variety $\mathscr V$ of topological algebras is congruence-permutable, then, 
given any congruence $\sim$ on any algebra $A\in \mathscr V$, 
the $\sim$-saturation of any open subset of $A$ is open. 
\end{theoremB}

\begin{corollary*}
Given any topological algebra $A$ in a 
congruence-permutable variety, its topological quotient  by any congruence $\sim$ is defined, i.e., the 
operations of $A/{\sim}$ are continuous with respect to the quotient topology. Moreover, on the abstract quotient 
$A/{\sim}$,  there exists a unique topology with respect to which all operations of $\mathscr V$ are continuous 
and the canonical projection $A\to A/{\sim}$ is a continuous open homomorphism. 
\end{corollary*}

The simple proof of this corollary is similar to that of Corollary~2.2 in~\cite{Taylor}.

Other important consequences of the permutability of congruences are as follows. 

\begin{theoremC}[{\hspace{1sp}\cite[Corollary~2.7]{Taylor}}]
Every $T_0$ topological 
algebra in a congruence-permutable variety of topological algebras is Hausdorff. 
\end{theoremC}

(Curiously, the condition that all $T_0$ algebras in a full variety of topological algebras are Hausdorff is 
equivalent to the existence of certain terms satisfying certain identities in that variety~\cite{KS}.)

\begin{theoremD}[{see \cite[Theorem~11]{Maltsev1954}, \cite[Exercise~2.4.C\,(c)]{Engelking}, and  
\cite[Corollary~2.4]{Taylor}}] 
Let $\mathscr V$ be a congruence-permutable variety of topological algebras. Then, for any $A\in \mathscr V$ and 
any congruence $\theta$ on $A$, the following conditions are equivalent:
\begin{itemize}
\item 
the quotient algebra $A/\theta$ is Hausdorff;
\item 
all equivalence classes of $\theta$ are closed in $A$; 
\item 
the relation $\theta$ (that is, the kernel of the canonical homomorphism $A\to A/\theta$) is closed in $A\times 
A$.
\end{itemize}
Moreover, if $A$ is a subalgebra of $B\in \mathscr V$, then the closure $\overline \theta$ of $\theta$ in 
$B\times B$ is a congruence on the closure of $A$ in $B$.
\end{theoremD}

In what follows, we use the standard symbols $\mathbb R$, $\mathbb Q$, $\mathbb N$, and $\mathbb N_0$ for the 
sets of real numbers, rational numbers, positive integers, and nonnegative integers, respectively. By $|X|$ we 
denote the cardinality of a set $X$. Details concerning notions and facts related to topology can be found 
in \cite{Engelking} and to algebra, in \cite{Gratzer}. 

\section{Universal algebras}
\label{sec1}

Let $A$ be any nonempty set. Given a nonnegative integer $n$, a map $A^n \to A$ is called an 
\emph{$n$-ary operation} on $A$. In particular, a \emph{nullary operation} is a map $\{\varnothing\}\to A$, 
which is  identified with a fixed element of~$A$. A \emph{universal algebra}, or simply \emph{algebra}, is a 
nonempty set $A$ together with any family of operations on~$A$. A more formal definition is as follows. 

Suppose given a set $ \Sigma$ (of symbols of operations) together with a function $\nu\colon \Sigma\to \mathbb 
N_0$ assigning arity to each symbol in $\Sigma$. For each $n\ge 0$, let $\Sigma_n$ denote the set of symbols in 
$\Sigma$ to which the arity $n$ is assigned, so that $\Sigma=\bigcup\{\Sigma_n: n\ge 0\}$. In the context of 
universal algebras, such a set $\Sigma$ together with $\nu$ is called a \emph{signature}, a \emph{type}, or a 
\emph{language}. A \emph{$\Sigma$-algebra structure} on $A$ is a family of maps $\Sigma_n\to A^{A^n}$, $n\ge 0$, 
which associate each symbol $\sigma\in \Sigma_n$ with an operation $\sigma^A\colon A^n\to A$ on $A$ for every 
$n$; the operations $\sigma^A$ are said to be \emph{basic}. The set $A$ with this structure is called a 
\emph{$\Sigma$-algebra} or a (\emph{universal}) \emph{algebra of signature} (or \emph{type}) $\Sigma$. In what 
follows, when there is no danger of confusion, we often omit the superscript $A$ and use the same notation for 
elements of $\Sigma$ (that is, symbols of operations) and the corresponding operations on the algebra. 

Given $\Sigma$-algebras $A$ and $B$, a map $f\colon A \to B$, and $\sigma\in \Sigma_n$, we say
that $f$ is \emph{compatible} with $\sigma$ if 
$$
\sigma^B(f(x_1), \dots, f(x_n))= f(\sigma^A(x_1,\dots, x_n))
$$
for any $x_1, \dots, x_n\in A$. A map compatible with all $\sigma\in \Sigma$ is called a \emph{homomorphism}. A  
homomorphism  with an inverse which is also a homomorphism is called an \emph{isomorphism}. The \emph{kernel} 
of a homomorphism $f\colon A\to B$ is defined by 
$$
\ker h= \{(x,y): (x,y)\in A^2, f(x)=f(y)\}.
$$

In universal  algebra a crucial role is played by special equivalence relations called congruences. In what 
follows, given an equivalence relation $\sim$ on a set $X$ and an $x\in X$, by $[x]_{\sim}$ we denote the 
equivalence class of $\sim$ containing $x$. 

An equivalence relation $\sim$ on a $\Sigma$-algebra $A$ is \emph{compatible} with an operation $\sigma\in 
\Sigma_n$ if $\sigma(x_1, \dots, x_n)\sim \sigma(y_1,\dots, y_n)$ whenever $x_i\sim y_i$ for all $i\le n$. An 
equivalence relation compatible with all operations is called a \emph{congruence}. It is easy to see that the 
quotient set of $A$ by any congruence $\sim$ is a $\Sigma$-algebra with operations defined by 
$\sigma([x_1]_{\sim}, \dots, [x_n]_{\sim})= [\sigma(x_1), \dots,\sigma(x_n)]_{\sim}$, where $\sigma \in 
\Sigma_n$, $n\ge0$, the canonical projection is a homomorphism, and $\sim$ is the kernel of this homomorphism. 
Moreover, the kernel of each homomorphism $h\colon A\to B$ is a congruence and $B$ is isomorphic to the quotient 
of $A$ by this congruence (see, e.g.,~\cite{Gratzer}). 

For any signature $\Sigma$ and any set $S$ of symbols (representing variable symbols), \emph{terms} or 
\emph{words} over $S$ are formal expressions composed of symbols in $\Sigma$ and $S$. Formally, terms and their 
variables are defined recursively as follows: 
\begin{itemize} 
\item 
each $x\in S$ is a term, and its variable is $x$; 
\item 
if $\sigma\in \Sigma_n$ and  
$t_1, \dots, t_n$ are terms, then $\sigma(t_1,\dots, t_n)$ is a term, and its variables are all variables of 
$t_1, \dots, t_n$. In particular, any $\sigma\in \Sigma_0$ is a term with no variables. 
\end{itemize} 

The notion of a \emph{subalgebra} of a given $\Sigma$-algebra is naturally defined as a subset $B$ of $A$ such 
that $\sigma(B^n)\subset B$ for any $n\ge 0$ and $\sigma \in \Sigma_n$. A $\Sigma$-algebra  $A$ is 
\emph{generated} by a set $X\subset A$ if any subalgebra of $A$ containing $X$ coincides with~$A$. 

An \emph{identity} over $\Sigma$ is a formula  $\forall x_1\forall x_2\dots\forall x_n (t=s)$, where 
$x_1,\dots,x_n$ are symbols of variables and $t$ and $s$ are terms all of whose variables are among 
$x_1, x_2, \dots, x_n$. 

A \emph{variety} is any class of algebras of the same signature closed under the formation of Cartesian products, 
subalgebras, and quotients. According to Birkhoff's theorem, a class $\mathscr V$ of $\Sigma$-algebras of the 
same signature is a variety if and only if there exists a set of identities over $\Sigma$ such that $\mathscr V$ 
consists of precisely those $\Sigma$-algebras in which all identities in this set hold.

A \emph{free algebra} in 
$\mathscr V$ on a set $X$, or an algebra \emph{freely generated} by $X$ in $\mathscr V$, is an algebra $A(X)\in 
\mathscr V$ with the following \emph{universal property}: Given any $A\in \mathscr V$ and any map $f\colon X\to 
A$, there exists a unique homomorphism $A(X)\to A$ extending $f$.  It is well known that $A(X)$ is indeed 
generated by $X$, that all free algebras on $X$ are isomorphic, and that every algebra in $\mathscr V$ generated 
by $X$ is a quotient of~$A(X)$. 

We denote the variety of all $\Sigma$-algebras by $\mathscr U(\Sigma)$.

Any nonempty set $X$ generates the \emph{absolutely free $\Sigma$-algebra} $W(X)$ (known also as \emph{term} or 
\emph{word algebra}) consisting of all terms over $ \Sigma$ with variables in $X$. This algebra is the free 
algebra on $X$ in the variety $\mathscr U(\Sigma)$ of all $\Sigma$-algebras (so that any $\Sigma$-algebra 
generated by $X$ is its quotient). Therefore, each term $t$ with $k$ variables naturally defines a $k$-ary 
operation $t^A\colon A^k\to A$ on any $\Sigma$-algebra $A$. We refer to such operations as \emph{derived 
operations} or \emph{polynomials}. Fixing all but one variables of a basic operation, we obtain a map 
$A\to A$; following Mal'tsev \cite{Maltsev1954}, we refer to such maps as \emph{principal translations} and to 
their finite compositions as \emph{elementary translations}. All elementary translations form a semigroup under 
composition.

Given a set $\Sigma'$ of derived operations, we can treat any $\Sigma$-algebra 
$A$ as a $\Sigma'$-algebra $A'$ in a natural way (it coincides with $A$ as a set but has different basic 
operations); the algebra $A'$ is called a \emph{derived $\Sigma'$-algebra} of $A$, and the variety generated by 
all derived $\Sigma'$-algebras of all $\Sigma$-algebras in a given variety $\mathscr V$ is called the 
\emph{derived variety} of~$\mathscr V$.

\section{Varieties of topological algebras}
\label{sec2}

A \emph{topological $\Sigma$-algebra} is a $\Sigma$-algebra $A$ with a topology with respect to which all 
operations $\sigma\colon A^n\to A$, where  $n\ge 0$ and  $\sigma \in \Sigma_n$, are continuous. A 
\emph{topological subalgebra} of a topological algebra $A$ is a subalgebra endowed with the induced topology, and 
the \emph{topological quotient} of $A$ by a congruence $\sim$ is $A/\sim$ with the quotient topology (provided 
that the operations are continuous with respect to this topology). If the underlying abstract algebra of $A$ 
belongs to a congruence-permutable variety, then, by Theorem~B, the topological quotient of $A$ by any congruence 
exists and the canonical projection is a continuous open homomorphism. 

There are several approaches to defining varieties of topological algebras: they may be defined as 
classes of topological algebras closed under certain topologo-algebraic 
operations, as classes defined by some topological analogues of identities, or in some hybrid way (for example, 
as classes of topological algebras which certain topological properties whose underlying abstract algebras 
belong to a certain class of abstract algebras). 

The first approach was developed by Taylor; in what follows, we use his definitions~\cite{Taylor}. 

\begin{definition}
A \emph{variety of topological algebras}, or a \emph{topological variety}, is a class of topological algebras 
of the same signature closed under the formation of topological products, topological quotients, and subalgebras. 
A variety of topological algebras closed under the formation of any continuous homomorphic images is said to be 
\emph{wide}. Varieties of topological algebras definable by identities (i.e., those consisting of \emph{all} 
topological algebras whose underlying abstract algebras belong to a given abstract variety of algebras) are 
called \emph{full varieties} of topological algebras. 
\end{definition}

Note that not every wide variety is full (for example, the variety of precompact topological groups is wide 
but not full). However, for any variety $\mathscr V$ of topological $\Sigma$-algebras, 
$$ 
\overline{\mathscr V}= \{A\in \mathscr U(\Sigma): \text{there exists a topology $\tau$ on $A$ such that 
$(A,\tau)\in \mathscr V$}\} 
$$ 
is a variety of $\Sigma$-algebras~\cite[Corollary~0.5]{Taylor}. This Taylor's observation, together with other 
results of Taylor and Mal'tsev, imply Theorems~A and~B.

\begin{proof}[Proof of Theorems~A and~B]
Whether or not the congruences of a topological algebra are permutable does not depend on the topology of the 
algebra. Thus, a variety $\mathscr V$ of topological algebras is congruence-permutable if and only if so is the 
variety $\overline{\mathscr V}$ of abstract algebras. By Mal'tsev's theorem in \cite[Theorem~4]{Maltsev1954} 
the congruence permutability of $\overline{\mathscr V}$ is equivalent to the existence of a Mal'tsev term in 
the language of $\overline{\mathscr V}$ (that is, of $\mathscr V$), which proves Theorem~A. 

By Taylor's theorem in \cite[Theorem~2.1]{Taylor} if the full variety of all topological algebras with underlying 
abstract algebras in $\overline{\mathscr V}$ (or, equivalently, the variety $\mathscr V$) is congruence-permutable, 
then, given any topological algebra $A\in \overline{\mathscr V}$ (in particular, any $A\in \mathscr V$), any 
open set $U\subset A$, and any congruence $\sim$ of $A$, the $ \sim$-saturation of $U$ is open. This proves 
Theorem~B. 
\end{proof}

Taylor proved in~\cite[Theorem~2.1]{Taylor} that, for full varieties, the converse of Theorem~B is also true; 
namely, if $\mathscr V$ is a full variety, then the saturation with respect to a congruence 
of any open set in any 
topological algebra $A\in \mathscr V$ is open if and only if $\mathscr V$ is congruence-permutable.

In the same paper~\cite{Taylor} Taylor tried to fit his definitions in the framework of the second approach to 
varieties of topological algebras (based on identities). He 
proved that wide varieties are precisely the classes of topological algebras defined by what he called limit 
identities and varieties are classes defined by limit identities and contingent limit 
identities\footnote{Contingent limit identities are a natural analogue of quasi-identities; thus, it 
might be more correct to refer to Taylor's varieties as quasivarieties.}. Later this approach was developed 
further by Protasov~\cite{Protasov}, who proposed the most 
general and natural definition of a variety of topological $\Sigma$-algebras as the class of 
$\Sigma$-algebra $A$ on which all filters in a certain set of filters on the absolutely free $\Sigma$-algebra 
$W(X)$ converge under any interpretation $X\to A$. In other words, he treated filters as topological identities 
and the convergence of filters as the fulfillment of identities. Protasov's varieties are precisely 
classes of algebras closed under the formation of topological products, closed subalgebras, and 
continuous homomorphic images; thus, both Taylor's full and wide varieties, as well as many other classes of 
topological algebras, are varieties in Protasov's sense, but Taylor's varieties are not. 

The third approach is exemplified by Mal'tsev's primitive classes~\cite{Maltsev1957} (consisting of all Hausdorff 
algebras in a given full variety) and Choban's $Q$-quasivarieties\footnote{This is Choban's term (the conventional 
meaning of the term ``quasivariety'' in universal algebra is quite different). In fact, such classes are more like 
prevarieties.}~\cite{Choban1985, Choban1993} (classes of topological algebras with a topological property $Q$, 
e.g., satisfying a certain separation axiom, and closed under the formation of topological products and 
subalgebras). 

In topological varieties  free topological algebras are defined by analogy with the abstract case.

\begin{definition}\label{d2}
Let $\mathscr V$ be a variety of topological algebras, and let $X$ be a topological space. A  
\emph{free topological algebra on $X$ in $\mathscr V$} is a topological algebra $A(X)\in \mathscr V$ together 
with a continuous map $i\colon X\to A(X)$ such that 
\begin{itemize} 
\item 
$A(X)$ is generated by $i(X)$; 
\item 
given any continuous map $f\colon X\to B$, 
where $B\in \mathscr V$, there exists a continuous homomorphism $h\colon A(X)\to B$ for which $f=h\circ 
i$. 
\end{itemize}
\end{definition}

It is well known that, given any topological variety $\mathscr V$ and any space $X$, a free topological algebra 
$A(X)$ on $X$ in $\mathscr V$ exists and is unique in the sense that if $A'(X)$ and $i'\colon X\to A'(X)$ satisfy 
the same conditions as $A(X)$ and $i$, then there exists a topological isomorphism $\varphi\colon A(X)\to A'(X)$ 
for which $i'=\varphi\cdot i$ (see, e.g., \cite{Maltsev1957}).

The notion of a free topological algebra is particularly useful in the case of a congruence-permutable 
topological variety, because, in this case, each topological algebra in the variety is a topological quotient of 
a free algebra.

A standard argument based on the diagonal of a set of maps from $X$ to algebras in $\mathscr V$ shows that such a 
pair always exists and is unique up to topological isomorphism (see \cite{Maltsev1957}). 

\begin{remark}\label{r1}
If $\mathscr V$ contains a non-one-element algebra, then 
\begin{enumerate}
\item[(i)]
$i$ is injective and 
\item[(ii)]
$A(X)$ is freely 
generated by $i(X)$, that is, it is isomorphic to the abstract free algebra $A(X)$ on 
$X$ 
\end{enumerate}
(see~\cite[Proposition~1.2]{Taylor}); this follows from the observation that, given any $A\in \mathscr V$, 
the algebra isomorphic to $A$ but endowed with the indiscrete topology belongs to $\mathscr V$~\cite{Taylor}. 
\end{remark}

\'Swierczkowski proved in~\cite{Swierczkowski}\footnote{Note that by a stronger topology \'Swierczkowski means a 
coarser one.} (see also \cite{Porst} and \cite{Porst-err}) that if $X$ is Tychonoff and the variety $\mathscr V$ 
is full, then, in addition, 
\begin{enumerate}
\item[(iii)]
$i$ is a homeomorphic embedding, so that $X$ can be identified with its image $i(X)$ 
in $A(X)$, 
\item[(iv)] 
the topological algebra $A(X)$ contains $X$ as a closed subspace, and 
\item[(v)] 
$A(X)$ is functionally Hausdorff; moreover, there exists a coarser Tychonoff topology $\tau$ on $A(X)$ such that 
$(A(X),\tau)$ still remains a topological algebra and contains $X$ as a closed subspace. 
\end{enumerate}

\section{Topological Mal'tsev Algebras}
\label{sec3}

We use $\mu$ as a symbol for a ternary operation.  

\begin{definition}
A \emph{Mal'tsev topological algebra} is a topological algebra with  
signature $\{\mu\}$ in which the equations 
$$
x=\mu(x,y,y)=\mu(y,y,x)\eqno(\mathrm M)
$$ 
hold identically. Any ternary operation $\mu$ satisfying the identity $(\mathrm M)$ is called a \emph{Mal'tsev 
operation}. 

By $\mathscr M$ we denote the variety of all Mal'tsev topological algebra.
	
A topological space on which a continuous Mal'tsev operation can be defined (i.e., a space homeomorphic to a 
Mal'tsev topological algebra) is called a \emph{Mal'tsev 
space}. 
\end{definition}

In what follows, by a Mal'tsev algebra we always mean a Mal'tsev topological algebra. 

According to Theorem~A, the underlying space of any topological algebra in any congruence-permutable variety 
of topological algebras is Mal'tsev. Well-known examples of such varieties are the variety of 
topological groups and the more general varieties of topological left loops and  quasigroups.

In a topological group, the Mal'tsev term is $\mu(x,y,z)=x\cdot y^{-1}\cdot z$.

A left loop is an algebra $L$ with two binary operations $*$ and $\setminus $ and one nullary operation 
(constant) $e$ satisfying the identities 
$$
x*(x\setminus y)= y,\qquad x\setminus (x*y)= y, \qquad x*e=x.
$$
A Mal'tsev term on $L$ is $\mu(x,y,z)=x*(y\setminus z)$. Indeed, $\mu(x,x,y) =y$ by the first identity and 
$\mu(y,x,x)=y*e=y$, because $x\setminus x=x\setminus (x*e)=e$ by the second identity.

A quasigroup is an algebra $Q$ with three binary operations $*$, $\slash$, and $\setminus$ 
satisfying the identities 
$$
(y\slash x)*x = y,\qquad (y*x)\slash x=y, \qquad x*(x\setminus y)=y,\qquad x\setminus (x*y)=y.
$$
A Mal'tsev term is $\mu(x,y,z) = (x\slash (y\setminus y))*(y\setminus z)$ (the identity $\mu(x,x,y)=y$ follows 
from $x\slash (x\setminus x)= (x*(x\setminus x))\slash (x\setminus x)=x$). 

\begin{proposition}
\label{p1}
\begin{enumerate}
\item[\rm(1)]
All $T_0$ Mal'tsev algebras are Hausdorff. 
\item[\rm(2)] 
All $T_0$ topological quasigroups are regular.
\item[\rm(3)] 
All $T_0$ topological left loops are regular.
\end{enumerate}
\end{proposition}

\begin{proof}
Assertion (1) is an immediate corollary of Theorem~C. 

(2)\enspace According to Theorem~2 of \cite{Maltsev1954} (see also the lemma preceding this theorem), for any 
quasigroup, there exists a derived algebra whose elementary translations form a transitive group, and by 
Theorem~7 of \cite{Maltsev1954}, any Hausdorff topological algebra with this property is regular. It remains to 
apply assertion (1) and the above observation that the underlying space of any topological quasigroup is 
Mal'tsev. 

(3)\enspace Recall that a topological space $X$ is said to be \emph{rectifiable} if there exists a homeomorphism 
$\psi\colon X^2\to X^2$ and a point $e\in X$ such that $\pi_1(\psi(x,y))=x$ and $\psi(x,x)=(x,e)$ for all $x\in 
X$ (here $\pi_1$ denotes the projection onto the first coordinate). It is easy to see that any topological left 
loop is rectifiable: it suffices to set $\psi(x,y)=(x, x\setminus y)$. The converse is also true (see, e.g., 
\cite{Usp}). According to Corollary~2.2 of \cite{Gulko}, any Hausdorff rectifiable space is regular. 
\end{proof}

Hausdorff Mal'tsev spaces are not necessarily regular. 

\begin{example}
Let $X$ be the set $\mathbb Q$ endowed with the topology whose base consists of all 
open intervals and all sets of the form $(a,b)\setminus \{\frac 1n:n\in \mathbb N\}$. Thus, the topology of $X$ 
is finer than the standard zero-dimensional topology of $\mathbb Q$ (induced from $\mathbb R$). 
According to Theorem~2 of \cite{3}, $X$ is a Mal'tsev space, and it is  Hausdorff but not regular. 
\end{example}

\section{Free Topological Mal'tsev Algebras}
\label{sec4}

We refer to the free topological algebra in 
$\mathscr M$ on a topological space $X$ as the \emph{free topological Mal'tsev algebra}, or simply the 
\emph{free Mal'tsev algebra}, on $X$ and denote it by $M(X)$. 
By definition, $M(X)$ is a Mal'tsev algebra for which there exists a continuous map $i_X\colon X\to M(X)$ 
with the following properties: 
\begin{itemize}
\item
$M(X)$ is generated by $i_X(X)$; 
\item 
given any continuous map $f\colon X\to M$, 
where $M\in \mathscr M$, there exists a continuous homomorphism $h\colon M(X)\to M$ such that 
$f=h\circ i_X$. In particular, any continuous map $i_X(X)\to M\in \mathscr M$ extends to a continuous homomorphism 
$M(X)\to M$. 
\end{itemize} 

\begin{theorem}
\label{t1}
\begin{enumerate}
\item[\rm(1)]
The free Mal'tsev algebra $M(X)$ exists for every topological space~$X$, and it is unique up to topological 
isomorphism. 
\item[\rm(2)] 
The map  $i_X\colon X\to M(X)$ is injective, and $M(X)$ is freely generated by~$i_X(X)$. 
\item[\rm(3)] 
Every Mal'tsev algebra $M$ is the image of $M(M)$ under an open continuous homomorphism which is 
simultaneously a retraction. Therefore, every Mal'tsev algebra is a topological quotient of a free Mal'tsev 
algebra. 
\item[\rm(4)] 
For any continuous map $f\colon X\to Y$, there exists a unique continuous homomorphism $h\colon 
M(X)\to M(Y)$ for which $i_Y\circ f = h\circ i_X$, and  if $i_Y\circ f\circ i_X^{-1}\colon 
i_X(X)\to i_Y(Y)$ is quotient, then $h$ is open. 
\item[\rm(5)] 
If $M(X)$ is $T_0$ (and hence Hausdorff), then $i(X)$ is closed in $M(X)$.
\item[\rm(6)] 
If $X$ is Tychonoff, then $X$ is embedded in $ M(X)$ as a closed subspace, and if $X$ is functionally Hausdorff, 
then so is $M(X)$.
\end{enumerate} 
\end{theorem}

\begin{proof}
Assertion (1) was proved by Mal'tsev in~\cite{Maltsev1957} for free topological algebras in any full varieties, 
and (2) follows from Proposition~1.2 of \cite{Taylor}. 

Let us prove (3). Take any Mal'tsev algebra $M$.  By definition, there exists a continuous homomorphism $h\colon 
M(M)\to M$ such that the identity homeomorphism $f\colon M\to M$ equals $h\circ i_M$. Therefore, $i_M\colon M\to 
M(M)$ is a homeomorphic embedding and $h$ is a retraction. Hence $h$ is quotient (it suffices to note 
that, for any $X\subset M$, we have $h^{-1}(X)\cap i_M(M)=i_M(X)$). Since the variety $\mathscr M$ is 
congruence-permutable and $h$ is continuous, it follows by Theorem~B that, given any open set $U\subset M(M)$, 
the set $h^{-1}(h(U))$ is open in $M(M)$ and hence $h(U)$ is open in $M$. Therefore, $h$ is an open 
continuous homomorphism and $M$ is the topological quotient of $M(M)$ by its kernel. 

Let us prove (4). Given a continuous map $f\colon X\to Y$, consider a  
continuous homomorphism $h\colon M(X)\to M(Y)$ for which $i_Y\circ f=h\circ i_X$. It is unique, because $M(X)$ is 
generated by $i_X(X)$ and, for any term $t$ in the language of $\mathscr M$ and any $x_1, \dots, x_n\in X$ (where 
$n$ is the number of variables of $t$), we have $h(t(x_1,\dots, x_n))=t(h(x_1),\dots, h(x_n))$. 

Now suppose that $i_Y\circ f\circ i_X^{-1}$ is quotient (and hence surjective). Then $h$ is surjective, because  
$M(Y)$ is generated by $i_Y(Y)$. Let $\tilde M(X)$ be the topological quotient of $M(X)$ by $\ker h$ (this is 
a Mal'tsev topological algebra by the corollary of Theorem~B), and let $\tilde h$ be the corresponding canonical 
projection. Then the identity map $\tilde i\colon \tilde M(X)\to M(Y)$ is a continuous isomorphism and $h = 
\tilde i\circ \tilde h$. The map $i_Y\circ f\circ i_X^{-1}$ coincides with the restriction of $h$ to $i_X(X)$. As 
a map of sets without topology, it also coincides with the restriction of $\tilde h$ to $i_X(X)$ (provided that 
we identify $\tilde M(X)$ and $M(Y)$ as sets). Since the map $i_Y\circ f\circ i_X^{-1}$ is quotient, it follows 
that the topology on its image is finest among those with respect to which it is continuous. Therefore, $\tilde 
i^{-1}|_{i_Y(Y)}\circ i_Y\colon Y\to \tilde M(X)$ is a continuous injection. The continuous homomorphism $g\colon 
M(Y)\to \tilde M(X)$ for which $\tilde i^{-1}|_{i_Y(Y)}\circ i_Y = g\circ i_Y$ must coincide with $\tilde i^{-1}$ 
as a map of sets,  because $M(Y)$ is freely generated by $i_Y(Y)$ and $i_Y$ is injective. Thus, the identity 
isomorphism $i$ is a homeomorphism and the homomorphism $h$ is quotient. By Theorem~B it is also open. 

We proceed to (5). Suppose that $M(X)$ is Hausdorff. Let us denote $M(X)$ by $M$ and consider the free Mal'tsev 
algebra $M(M)$. It follows from (3) that $M$ can be treated as a closed subspace of $M(M)$ (because any retract of 
a Hausdorff space is closed). Let $M_X$ be the subalgebra of $M(M)$ generated by $i_X(X)\subset M$. Clearly, this 
is algebraically the free algebra on $i_X(X)$ (because we can extend any map $f$ from $i_X(X)$ to a 
$\{\mu\}$-algebra $A$ first to some map $M\to A$ and then to a 
homomorphism $M(M)\to A$, whose restriction to $i_X(X)$ is a homomorphism $M_X\to A$ extending $f$). 
Note that $M_X\cap M=i_X(X)$. Indeed, suppose that $y_0\in M\setminus i_X(X)$ and consider 
the map (not necessarily continuous) $g\colon M\to \{0,1\}$ defined by setting $g(y)= 1$ if $y=y_0$ and $g(y)= 0$ 
otherwise. The set $\{0,1\}$ carries a group structure and hence a Mal'tsev operation. Therefore, the map $g$ 
extends to a homomorphism of the abstract free $\{\mu\}$-algebra on $M$ (which coincides with $M(M)$ as a set) to 
$\{0,1\}$. The preimage  of $0$ under this homomorphism contains the subalgebra of $M(M)$ generated by $i(X)$ and 
does not contain $y_0$. 

Thus, $M_X$ and $M(X)$ coincide as abstract algebras and $M_X$ with the topology induced on $M_X$ from $M(M)$ 
is a Mal'tsev space containing $i_X(X)$ as a closed subspace. The identity homeomorphism $i_X(X)\to i_X(X)$ 
extends to a continuous homomorphism $M_X\to M(X)$, which is an isomorphism, because both algebras $M_X$ and 
$M(X)$ are freely generated by $i_X(X)$. This means that there exists a topology on $M(X)$ which is coarser than 
the topology of $M(X)$ and with respect to which $i_X(X)$ is closed. Therefore, $i_X(X)$ is closed in $M(X)$. 

Both assertions of (6) follow from \'Swierczkowski's theorem~\cite{Swierczkowski}. 
\end{proof}

According to \'Swierczkowski's theorem~\cite{Swierczkowski}, any Tychonoff space $X$ is embedded in $M(X)$ as a 
closed subspace. However, this is not always the case, at least because there exists a topological space 
that cannot be embedded in a congruence-permutable topological algebra as a 
subspace~\cite[Corollary~3.6]{Coleman}. The following 
question naturally arises.

\begin{problem*}
What topological spaces $X$ are embedded in $M(X)$ as (closed) subspaces?
\end{problem*}

Together with the variety $\mathscr M$ of all Mal'tsev algebras, it makes sense to consider the class 
of all Tychonoff Mal'tsev algebras. Let us denote it by $\mathscr M^{\tych}$. We define the \emph{free Tychonoff 
Mal'tsev algebra} $M^\tych(X)$ on a Tychonoff space~$X$ as a Tychonoff Mal'tsev algebra for 
which there exists a continuous map $i_X\colon X\to M^\tych(X)$  with the following properties: 
\begin{itemize} 
\item $M^\tych(X)$ is generated by $i_X(X)$; 
\item given any continuous map $f\colon X\to M$, where $M\in 
\mathscr M^\tych$, there exists a continuous homomorphism $h\colon M^\tych(X)\to M$ such that $f=h\circ i_X$. 
\end{itemize} 

\begin{theorem} 
\label{t2}
Let $X$ be any Tychonoff space.
\begin{enumerate}
\item[\rm(1)]
The free Mal'tsev space $M^\tych(X)$ exists and is unique up to topological 
isomorphism. 
\item[\rm(2)]
The map $i_X$ is a topological embedding, so that $X$ can be identified with the subspace $i_X(X)$ of  
$M^\tych(X)$. Moreover, this subspace is closed.
\item[\rm(3)]
The Mal'tsev space $M^\tych(X)$ is freely generated by $i_X(X)=X$. 
\item[\rm(4)]
Every $M\in \mathscr M^\tych$ is the image of 
$M^\tych(M)$ under an open continuous homomorphism which is simultaneously a retraction. Therefore, every 
Tychonoff Mal'tsev space is a topological quotient of a free Tychonoff Mal'tsev space. 
\item[\rm(5)]
Any quotient map $X\to M$, where $M\in \mathscr M^\tych$, extends to a quotient homomorphism $M^\tych(X)\to M$. 
\end{enumerate} 
\end{theorem}

\begin{proof}
Assertion (1) is proved by a well-known standard argument based on the diagonal theorem, which applies to any 
multiplicative hereditary topological property rather to only being Tychonoff. 

Note that, in the definition of the free Mal'tsev space $M^\tych(X)$, it suffices to require that 
homomorphisms $h$ with the property $f=h\circ i_X$ exist only for continuous maps $f$ from $X$ to Tychonoff 
Mal'tsev algebras $M$ generated by $f(X)$. Any such $M$ has cardinality at most $|X|\cdot \omega$ and, therefore, weight 
at most $2^{|X|\cdot \omega}$. Clearly, it is also sufficient to consider the maps $f\colon X\to M$ up to composition 
with topological isomorphisms of $M$; thus, we can restrict ourselves to Mal'tsev algebras whose underlying spaces 
are contained in the Tychonoff cube $[0,1]^{2^{|X|\cdot \omega}}$ as subspaces. Clearly, all such Mal'tsev 
algebras, as well as continuous maps $f$ from $X$ to them, form a set. Let us index these maps by the elements of 
some set $I$ as $f_\iota\colon X\to M_\iota$, $\iota\in I$ (spaces $M_\iota$ with different indices may coincide). 
We set $i_X=\mathop{\Delta}\limits_{\iota\in I}f_\iota\colon X\to \prod_{\iota\in I}M_\iota$ 
and define $M^\tych(X)$ to be the Mal'tsev subalgebra of 
$\prod_{\iota\in I}M_\iota$ generated by $i_X(X)$. It is easy to see that $i_X$ and $M^\tych(X)$ have the required 
properties: any continuous map $f$ from $X$ to a Tychonoff Mal'tsev algebra $M$ coincides with $f_\iota$ for 
some $\iota\in I$ (up to a topological isomorphism $\varphi$ between $M_\iota$ and the subalgebra of $M$ generated 
by $f(X)$), and the required homomorphism $h$ is the composition of the restriction of the canonical projection 
$\pi_\iota$ to $i_X(X)$ and $\varphi$. Uniqueness follows from that if $M_1(X)$ and $M_2(X)$ are two free 
Tychonoff Mal'tsev algebras on $X$ and $i_j\colon X\to M_j(X)$ are the corresponding maps $i_X$, then the 
homomorphisms $h_1\colon M_1(X)\to M_2(X)$ and $h_2\colon M_2(X)\to M_1(X)$ for which $i_2=h_1\circ i_1$ and 
$i_1=h_2\circ i_2$ must be mutually inverse and hence they are topological isomorphisms (see also the proof of 
Theorem~1 in \cite{Maltsev1957}). 

Let us prove (2). According to \cite{Swierczkowski}, there exists a Tychonoff topology on the abstract free 
$\{\mu\}$-algebra $M^a(X)$ on the set $X$ such that the space $X$ is embedded in 
$(M^a(X), \tau)$ as a closed subspace. By definition, for the identity homeomorphism $f\colon X\to X$, there 
exists a continuous homomorphism $h\colon (M^a(X), \tau) \to M^\tych(X)$ such that $f= h\circ i_X(X)$. 
This means that $i_X$ is a homeomorphic embedding. Clearly, $h$ is an isomorphism. Therefore, $M^\tych(X)$ 
is $M^a(X)$ with a topology finer than $\tau$. Hence $X=i_X(X)$ is closed in $M^\tych(X)$. This 
also proves~(3). 

The proof of assertion (4) repeats the proof of assertion (3) of Theorem~\ref{t1}. 

Assertion (5) is proved along the same lines as assertion (4) of Theorem~\ref{t1}. Any quotient map $f\colon 
X\to M$ is surjective; therefore, so is its homomorphic extension $h\colon M^{\tych}(X)\to M$. Let $\tilde M(X)$ 
be the topological quotient of $M(X)$ by $\ker h$, and let $\tilde h$ be the corresponding canonical projection. 
Then the identity map $\tilde i\colon \tilde M(X)\to M$ is a continuous isomorphism and $h = \tilde i\circ \tilde 
h$. The map $f$ coincides with the restriction of $h$ to $X$. As a map of sets without topology, it also 
coincides with the restriction of $\tilde h$ to $X$ (provided that $\tilde M(X)$ and $M$ are identified as sets). 
Since the map $f$ is quotient, it follows that the topology on its image $M$ is finest among 
those with respect to which it is continuous. But $\tilde h|_X$ is continuous with respect to the topology 
of $\tilde M(X)$; therefore, the topology of $M$ is finer than that of $\tilde M(X)$, so that $\tilde i^{-1}$ is 
continuous and the identity isomorphism $\tilde i$ is topological. Hence the homomorphism $h = \tilde 
i\circ \tilde h$ is quotient. By Theorem~B it is also open. 
\end{proof}

Recall that the free topological group $F(X)$ of a Tychonoff\footnote{We assume $X$ to be Tychonoff only for 
simplicity; the considerations remain valid in the general case.} space $X$ (the free 
topological algebra on $X$ in the variety of all topological groups) can be represented as the set of reduced 
words $x_1^{\varepsilon _1}\dots x_n^{\varepsilon _n}$, where $n\in \mathbb N_0$ (if $n=0$, then the word is 
empty), $\varepsilon_i=\pm1$ and $x_i\in X$ for $i\le n$. It is assumed that a homeomorphism ${}^{-1}\colon X\to 
X^{-1}$ between $X$ and its disjoint homeomorphic copy is fixed, and for each $x\in X$, $x^{-1}$ denotes the 
image of $x$ under this homeomorphism. A word is reduced if it does not contain pairs of neighboring letters of 
the forms $xx^{-1}$ and $x^{-1}x$. Multiplication in this group is concatenation followed by reduction, that is, 
deleting all prohibited pairs $xx^{-1}$ and $x^{-1}x$. For example, $xyz\cdot z^{-1}y^{-1}x=xx$. The identity 
element is the empty word. 

The free topological group of any 
Tychonoff space exists and is Tychonoff.

Setting 
$$
F_k(X)=\{x_1^{\varepsilon _1}\dots x_n^{\varepsilon _n}: n\le k, x_i\in X, \varepsilon_i=\pm 1\},
$$ 
we obtain the decomposition 
$$
F(X)=\bigcup_{k\ge 0}F_k(X).
$$
Each $F_k(X)$ is the image of the space $(X\oplus \{e\} \oplus X^{-1})^k$ under the natural 
multiplication map $m_k$ defined by 
$$
m_k(x_1^{\varepsilon _1},\dots, x_n^{\varepsilon _n})= x_1^{\varepsilon _1}\cdot\dots \cdot x_n^{\varepsilon _n}
$$  
(here $\{e\}$ is a singleton disjoint from $X$ and $X^{-1}$; $e$ represents the identity element of $F(X)$) 
and $F(X)$ is the image of the topological sum $W_F(X)=\bigoplus_{k\ge 0}(X\oplus \{e\} \oplus X^{-1})^k$ under 
the multiplication map $m$ defined as $m_k$ on the corresponding summand. The topological sum is the absolutely 
free $\{e,  {}^{-1}, \cdot\}$-algebra and $m$ is the canonical projection $W(X)\to W(X)/{\sim}= F(X)$, where 
$\sim$ is the smallest congruence determined by the standard group identities (associativity and so on). Of 
course it is very important to know when the map $m$ is quotient, i.e., when $F(X)$ is the topological quotient 
of $W_F(X)$. This is the case if and only if all $m_k$ are quotient and $F(X)$  is the direct (=~inductive)  
limit of its subspaces $F_k(X)$, that is, a set $U\subset F(X)$ is open in $F(X)$ if and only if each 
intersection $U\cap F_k(X)$ is open in $F_k(X)$. The problem of describing all spaces $X$ for which $m$ is 
quotient, as well as the problems of describing $X$ for which $F(X)$ has the direct limit topology and $X$ for 
which the maps $m_k$ are quotient, is very difficult. So far, only a few sufficient conditions have been 
obtained. For example, it is known that if $X$ is the direct limit of a countable sequence of its compact 
subspaces (such spaces are called \emph{$k_\omega$-spaces}), then $m$ is quotient (see~\cite{MMO}). 

For free Mal'tsev spaces, there is a similar decomposition. First, note that the absolutely 
free $\{\mu\}$-algebra $W(X)$ on a topological space $X$ can be constructed by induction as follows: 
\begin{gather*}
W_0(X)=X,\\
W_1(X)= W_0(X)\times W_0(X)\times W_0(X),\\ 
\dots \\
W_n(X) =\hskip-6pt \bigoplus_{\substack{i,j,k\ge 0\\\hskip-9pt\max\{i,j,k\}=n-1\hskip-9pt}}\hskip -6pt 
W_i(X)\times W_j(X)\times W_k(X),\\
\dots\,.
\end{gather*}
Obviously, $W_i(X)\cap W_j(X)=\varnothing$ for $i\ne j$. We set 
$$
W(X)=\bigoplus_{i\in \mathbb N_0} W_i(X).
$$

The operation $\mu$ on $W(X)$ is defined by 
$$
\mu(x,y,x)=(x,y,z)\in W_i(X)\times W_j(X)\times W_k(X)\subset W_{\max\{i,j,k\}+1}(X).
$$
It is easy to see that the map $\mu\colon W(X)^3\to W(X)$ is continuous. 

Consider the relation $R$ on $W(X)$ defined by the rule: $(x,y)\in R$ if there exists an $z\in W(X)$ for which 
$x=(z,z,y)$ or $x=(y,z,z)$. Let $\sim$ be the smallest congruence containing $R$ (that is, the intersection of 
all such congruences). Algebraically, $M(X)$ is the quotient of $W(X)$ by $\sim$. Let us denote the canonical 
projection $W(X)\to W(X)/{\sim}=M(X)$ by $j$.

For $k\in \mathbb N_0$, we set 
$$
M_k(X)=j(W_k(X))\quad \text{and} \quad j_{k}=j|_{W_k}\colon W_k\to M_k.
$$
Note that, for each $w=(x,y,z)\in W_i(X)$ and any $x_0\in X$, we have $u=(w,x_0,x_0)\in W_{i+1}(X)$ and 
$j(w)=j(u)$. Therefore,  $M_i(X)\subset M_k(X)$ for $i\le k$ and $M_k$ is the set of all polynomials in at 
most $3^i$ variables in $M(X)$. In particular, $X=M_0(X)\subset M_1(X)$. 

\begin{proposition}
\label{p2}
A topological space $X$ is Mal'tsev if and only if $X$ is a retract of $M_1(X)$. 
\end{proposition}

\begin{proof}
Any Mal'tsev space $X$ is a retract of $M(X)$ by Theorem~\ref{t1}\,(3). Clearly, the restriction of a retraction 
to any subspace containing the retract is a retraction as well.  

Conversely, if $r\colon M_1(X)\to X$ is a retraction, then $r\circ j_1\colon W_1(X)\to X$ is a Mal'tsev 
operation (recall that $W_1(X)=X^3$). 
\end{proof}

For a Tychonoff $X$, the set $M_1(X)$ is very much like the subset $G_1(X)=\{x\cdot y^{-1}\cdot z\in F(X): 
x,y,z\in X\}$ of the free group $F(X)$. Moreover, considering $F(X)$ as a Mal'tsev space with the operation 
$\mu(u,v,w)=u\cdot v^{-1}\cdot w$ and extending the identity embedding $X\to F(X)$ to a continuous homomorphism 
$h\colon M(X)\to F(X)$, we see that $h|_{M_1(X)}\colon M_1(X)\to G_1(X)$ is a continuous bijection. It is 
defined by $h(j_1(x,y,z))= x\cdot y^{-1}\cdot z$ for $x,y,z\in X$. If this bijection were a homeomorphism, then we 
could assert that a topological space is Mal'tsev if and only if it is a retract of the space $G_1(X)\subset 
F(X)$. However, this is not always the case. 

\begin{example}
The following example was constructed in \cite{3}.
Let $C$ be the Cantor space, and let $\{M\}\cup\{M_\alpha: \alpha<2^\omega\}$ be a partition of $C$ into 
subspaces homeomorphic to $C$. To construct such a partition, it suffices to recall that $C=\{0,1\}^{\mathbb N}$, 
index all non--identically zero sequences $(x_n)_{n\in \mathbb N}\in C$ as $(x_n^\alpha)_{n\in \mathbb N}$, and 
set $M=\{(x_n)_{n\in \mathbb N}\in C: x_{2n}=0\text{ for $n\in \mathbb N$}\}$ and 
$M_\alpha=\{(x_n)_{n\in \mathbb N}\in C: x_{2n}=x^\alpha_n\text{ for $n\in \mathbb N$}\}$. Let us strengthen 
the standard topology of the Cantor space by declaring the sets $M_\alpha$ to be closed and open, and let us 
denote the Cantor set with this strengthened topology by~$X$. It was proved in \cite{3} that $X$ is a  
Mal'tsev space but is not a retract of $G_1(X)$. By Proposition~\ref{p2}, $X$ is a retract of 
$M_1(X)$. Therefore, $M_1(X)$ is not homeomorphic to $G_1(X)$. 
\end{example} 

\begin{proposition}
\label{p3}
If $X$ is a Tychonoff space and  the multiplication map 
$$
\bar m_3= m_3|_{X\times X^{-1}\times X}\colon X\times X^{-1}\times X \to G_1(X), \quad (x,y^{-1},z)\mapsto 
x\cdot y^{-1}\cdot z, 
$$ 
is quotient, then $M_1(X)$ is homeomorphic to $G_1(X)$. If $X$ is in addition Mal'tsev, 
then $X$ is a retract of $G_1(X)$. 
\end{proposition}

\begin{proof}
Note that the map $h\circ j_1\colon X\times X\times X\to G_1(X)$ coincides with the composition of the 
homeomorphism $\varphi\colon X\times X\times X\to X\times X^{-1}\times X$ defined by 
$\varphi(x,y,z)=(x,y^{-1},z)$ and the multiplication map $\bar m_3$. Thus, if $\bar m_3$ 
is quotient, then so is $h\circ j_1$ and $h|_{M_1(X)}$, which implies that $M_1(X)$ is 
homeomorphic to $G_1(X)$ in this case; moreover, the homeomorphism is compatible with the restriction of the 
natural ``free group'' Mal'tsev operation to $M_1(X)$. This, together with Proposition~\ref{p2}, immediately 
implies that if $X$ is a Mal'tsev space for which $\bar m_3$ is quotient, then $X$ is a retract of $G_1(X)$. 
\end{proof} 

In \cite{RS} it was proved that if the multiplication map $m_3\colon (X\oplus\{e\}\oplus X^{-1})^3\to F_3(X)$ is 
quotient, then so is~$\bar m_3$.  

The following example shows that the condition that $\bar m_3$ (or $m_3$) is quotient is only sufficient but not 
necessary for a Tychonoff Mal'tsev space $X$ to be a retract of $G_1(X)$. 

\begin{example}
It was shown in \cite{FOT} that the map $\bar m_3$ is not quotient for the space $\mathbb Q$ of 
rational numbers with the standard topology. However, $\mathbb Q$ is a group and hence a Mal'tsev space with the 
group Mal'tsev operation $\mu(p,q,r)=p-q+r$. The map $r_G\colon G_1(\mathbb Q)\to \mathbb Q$ defined by $r_G(p\cdot 
q^{-1}\cdot r)=p-q+r$ is a retraction. 
\end{example}

Note that if $X$ is a Hausdorff compact space, then so are all spaces $(X\oplus\{e\}\oplus X^{-1})^k$, 
$W_k(X)$, and hence $M_k(X)$. Moreover, in this case, $M(X)$ is Tychonoff, so that all $M_k(X)$ are closed 
subspaces of $M(X)$. This implies the following proposition. 

\begin{proposition}\label{p-a}
For any Tychonoff space $X$, all sets $M_n(X)$ are closed in~$M(X)$.
\end{proposition}

\begin{proof}
Let $bX$ be any Hausdorff compactification of $X$. Then the subalgebra $\tilde M(X)$ of $M(bX)$ generated by $X$ 
is a topological Mal'tsev algebra containing $X$ as a subspace. It coincides with $M(X)$ as an abstract 
algebra but has a coarser topology. Each set $M_n(X)$ coincides with $M_n(bX)\cap \tilde M(X)$; therefore, it is 
closed in $\tilde M(X)$ and, therefore, in $M(X)$.
\end{proof}

Note also that, in the case where $X$ is compact, the maps $m_k$ are $j_k$ are quotient and $M_1(X)$ is 
homeomorphic to $G_1(X)$, so that any compact Mal'tsev algebra is a retract of $G_1(X)$. Moreover, any 
compact Hausdorff Mal'tsev algebra is a retract of the whole free group $F(X)$. This is implied by the 
following general theorem. 

\begin{theorem}[{see \cite{retract} and \cite{3}}]
\label{t3}
If $X$ is a Tychonoff Mal'tsev space such that the free topological group $F(X)$ has the direct limit 
topology with respect to the decomposition $F(X)=\bigcup_{n\ge 0}F_n(X)$ and all multiplication maps $m_n$ are 
quotient (that is, $m$ is quotient), then $X$ is a retract of~$F(X)$. 
\end{theorem}

Topological spaces being retracts of topological groups are said to be \emph{retral}. 
As mentioned above, the assumptions of Theorem~\ref{t3} hold, e.g., for all Tychonoff 
$k_\omega$-spaces, that is, direct limits of countably many compact subspaces \cite{MMO}. Therefore, 
$k_\omega$-spaces are retral. Some other sufficient conditions for being retral are given in \cite{3}. 
Reznichenko and Uspenskii also proved that the Mal'tsev operation on a pseudocompact Mal'tsev space $X$ can be 
extended to a continuous Mal'tsev operation on $\beta X$, which implies that all pseudocompact Mal'tsev spaces 
are retral~\cite{RU}. 

Of course, it is also of interest when the free Mal'tsev algebra $M(X)$ is the topological quotient of the 
absolutely free algebra $W(X)$. In \cite[Lemma in \S\,5]{Maltsev1957} 
Mal'tsev essentially proved the following theorem: \emph{Suppose that $X$ is a $k_\omega$-space, that is, the 
direct limit of its compact subspaces $X_n$, and $A$ is a Hausdorff topological algebra of at most countable 
signature generated by $X$. Then $A$ decomposes into the union of its subspaces of the form 
$p(X_n,\dots,X_n)=\{p(x_1, \dots, x_k): x_i\in X_n\}$, where $p$ is a polynomial on $A$, and all operations of 
$A$ are continuous in the direct limit topology with respect to this decomposition.} It follows immediately that, 
for a Hausdorff $k_\omega$-space $X$, $M(X)$ is the direct limit of its subspaces $M_k(X_n)$ and hence of 
$M_n(X)$. Therefore, $M(X)$ is itself a $k_\omega$-space (which implies, in particular, that it is Tychonoff). It 
is also easy to show that the map $j$ is quotient in this case (because so are all maps $p\colon X_n\times \dots 
\times X_n\to p(X_n,\dots,X_n)$, being continuous maps of Hausdorff compact spaces). 

Apparently, generalizations of this theorem similar to those in the case of free topological groups can be 
proved, but nothing fundamentally new should be expected. To understand why, consider the sets 
$$
G_n(X)= \{x_1\cdot x_2^{-1}\cdot x_3\cdot \dots \cdot x_{2n}^{-1}\cdot x_{2n+1}: x_i\in X\}\subset 
F_{2n+1}(X), \quad n\ge 0,
$$ 
and $G(X)=\bigcup _{n\ge 0} G_n(X)$. There is a natural Mal'tsev operation on $G(X)$, which is defined by the 
rule $\mu(a,b,c)=a\cdot b^{-1}\cdot c$. Note that this operation satisfies the associativity-type condition
$$
\mu(\mu(a,b,c),d,e)=\mu(a,\mu(d,c,b),e)=\mu(a,b,\mu(c,d,e)),\eqno(*)
$$
so that $G(X)$ is a \emph{heap}. 
Heaps differ from groups only in that they have no fixed identity element. Choosing any element $x_*\in 
G(X)$ and setting $u*v=\mu(u,x_0,v)$, we obtain a multiplication $*$ on $G(X)$. The role of the identity element 
is played by $x_*$, and the element inverse to $u$ is $\mu(x_*, u, x_*)$. Thus, $G(X)$ with the three derived 
operations thus defined and the topology inherited from $F(X)$ is a topological group. 

\begin{proposition}
\label{p4}
Let $\mathscr V$ be a full topological subvariety of $\mathscr M$, i.e., a class of topological Mal'tsev algebras 
which is itself a full topological variety. Then, for any topological space $ X$, the free topological algebra 
$A(X)$ on $X$ in $\mathscr V$ is a topological quotient of $M(X)$. In particular, if $X$ is Tychonoff, then 
the heap $G(X)$ (with its Mal'tsev operation $\mu$ satisfying condition $(*)$) is a topological quotient 
of~$M(X)$. 
\end{proposition}

\begin{proof}
Let $i\colon X\to A(X)$ be the map in Definition~\ref{d2} (of the free topological algebra $A(X)$). The case 
where $\mathscr V$ contains only the one-element algebra is trivial; thus, we will assume that $\mathscr V$ 
contains non-one-element algebras, in which case $A(X)$ is freely generated by $i(X)$ (see Remark~\ref{r1}).   
Since $i$ is continuous and $A(X)\in \mathscr M$, it follows that there exists a continuous homomorphism $h\colon 
M(X)\to A(X)$ such that $i=h\circ i_X$. It is surjective, because $i_X$ is injective (by Theorem~\ref{t1}\,(2)) 
and $A(X)$ is generated by $i(X)$. Therefore, as an abstract algebra, $A(X)$ is the quotient $M(X)/\ker h$. 
According to the corollary of Theorem~B, the topological quotient $\tilde A(X)= M(X)/\ker h$ is a Mal'tsev 
topological algebra. It is isomorphic to $A(X)$ as an abstract algebra (and hence all identities defining the 
variety $\mathscr V$ hold in $\tilde A(X)$, so that $\tilde A(X)\in \mathscr V$), and its topology is finer than 
that of $A(X)$. Thus, the map $i$ treated as a map from $X$ to $\tilde A(X)$ remains continuous. Hence there 
exists a continuous homomorphism $f\colon A(X)\to \tilde A(X)$ for which $i=f\circ i$. Since $i$ is injective and 
both $A(X)$ and $\tilde A(X)$ are freely generated by $i(X)$, it follows that $f$ is the identity homomorphism. 
Its continuity implies that the topology of $A(X)$ is finer than that of $\tilde A(X)$, i.e., these topologies 
coincide, which means that $A(X)$ is a topological quotient of~$M(X)$.
\end{proof}

\begin{proposition}
If $X$ is a Tychonoff space for which $M(X)$ is the direct limit of its subspaces $M_n(X)$, $n\ge 0$, then $G(X)$ 
is the direct limit of its subspaces $G_n(X)$, $n\ge 0$. 
\end{proposition}

\begin{proof}
According to Proposition~\ref{p4}, $G(X)$ is the image of $M(X)$ under a quotient map. Now the required assertion 
follows from the simple fact that if a space $Y$ is the direct limit of its closed subspaces $Y_n$, $n\in 
\omega$, and $Z$ is the image of $Y$ under a quotient map $f$, then $Z$ is the direct limit of its subspaces 
$f(Y_n)$, $n\in \omega$. Indeed, given any set $A\subset Z$ such that $A\cap f(Y_n)$ is closed in $f(Y_n)$ for 
each $n\in \omega$, every intersection $f^{-1}(A)\cap Y_n= f^{-1}(f(A)\cap f(Y_n))\cap Y_n$ is closed in 
$f^{-1}(f(Y_n))\supset Y_n$ and hence in $Y_n$. Therefore, $f^{-1}(A)$ is closed in $Y$ and $A$ is closed in $Z$ 
(because $f$ is quotient).
\end{proof}

\begin{remark}
The heap $G(X)$ is freely generated by the set $X$, because this is algebraically the quotient of the free 
algebra $M(X)$ by the associativity relation. It is seen from Proposition~\ref{p4} and its proof 
that, for a Tychonoff space $X$,  $G(X)$ is the free topological heap on $X$. Indeed, if $G_f(X)$ is the 
free topological heap on $X$ (it exists by Theorem~\ref{t1}), then the identity embedding $X\to G(X)$ must 
extend to a continuous homomorphism $G_f(X)\to G(X)$, and this homomorphism must be an isomorphism, because 
both $G(X)$ and $G_f(X)$ are freely generated by $X$. However, it was shown in the proof of Proposition~\ref{p4} 
that the topology of $G(X)$ is strongest among all topologies with respect to which the Mal'tsev 
operation on $G(X)$ is continuous and $X$ is contained in $G(X)$ as a subspace. Therefore, $G_f(X)=G(X)$.
\end{remark}

\begin{remark}
A topological space $X$ is a retract of a topological group if and only if $X$ is a retract of $G(X)$. Indeed, 
$G(X)$ is homeomorphic to a topological group, being a topological heap; this implies sufficiency. To prove 
necessity, suppose that $G$ is a topological group and $r\colon G\to X$ is a retraction. Then $r(G)=X$ is a 
subspace of $G$, and since $G$ carries the derived associative Mal'tsev (=~heap) operation $\mu(u,v,w)=u\cdot 
v^{-1}\cdot w$, it follows that the identity homomorphism $X\to G$ extends to a $\mu$-homomorphism 
$h\colon G(X)\to G$. The composition $r\circ h$ is a retraction $G(X)\to X$. 
\end{remark}

\section*{Acknowledgments}

The authors are very grateful to Evgenii Reznichenko for fruitful discussions.

\end{document}